\documentclass{article}
\usepackage[utf8]{inputenc}
\usepackage{amsmath}
\usepackage{times}
\usepackage{float}
\usepackage{amssymb}
\usepackage{amsthm}
\newtheorem{proposition}{Proposition}
\title{On The Diameter of Pancake Graphs}
\author{Harigovind V R \and Pramod P Nair}
\date{}
\begin{document}
\maketitle
\section*{Abstract}
The Pancake graph($P_n$) represents the group of all permutations on n elements, namely $S_n$, with respect to the generating set containing all prefix reversals. The diameter of a graph is the maximum of all distances on the graph, where the distance between two vertices is the shortest path between them. In the case of the $P_n$, it is the maximum of the shortest generating sequence of each permutation in $S_n$. Here we propose a method to realise better upper bounds to the diameter of $P_n$ that has its focus on Graph Theoretical concepts rather than Algebra.\\
\textbf{\small{Keywords:}}\small{Pancake Graph, Diameter, Prefix Reversals, Path Covering Spanning trees.}\\
\textbf{\small{Subject Classification:}}\small{05C12, 05C25, 05C85}
\section*{Introduction}
A Permutation $\pi$ represents a bijection from a set $A$ onto itself. It is written as $\pi = [\pi_{1} \pi_{2} ... \pi_{n}]$ where $\pi_{i}$ is the image of the $i^{th}$ element. If the set $A=\{1, 2, 3, ... ,n\}$ then the group of all such permutations is called the Symmetric group of degree n ($S_n$). The Pancake graph $P_n$  represents the graph for the group $S_n$ with respect to the generating set of all prefix reversals, $S=\{r_{2}, ... r_{n}\} $ where:
\begin{equation}
    [\pi_{1} \pi_{2} ... \pi_{i-1} \pi_{i} ... \pi_{n}] r_{i} = [\pi_{i} \pi_{i-1} ... \pi_{2} \pi_{1} ... \pi_{n}]
\end{equation}
All permutations in $S_n$ can be represented as a product of the elements of the generating set $S$, Hence $P_n$ is a connected graph. Furthermore, the degree of each vertex of $P_n$ is equal to the cardinality of S. The absence of the identity permutation in set $S$ excludes the possibility of loops in $P_n$. Hence $P_n$ is a simple connected graph which is $(n-1)$ regular. Hence in total we see that $P_n$ is a simple connected regular graph. A graph $G$ is said to be vertex transitive if for each pair of vertices in $G$ there exists some automorphism taking one vertex to the other. In case of $P_n$ we attain this idea through the automorphism group of $S_n$.\par 
\noindent Analyzing the structure of $P_n$ shows us that, for each $n$ the graph contains $n$ copies of $P_{n-1}$. This is the hierarchical property of the structure of $P_n$. The property plays an important role in proving the Hamiltonicity of $P_n$ for $n > 3$. The graph $P_{3}$ is isomorphic to the graph $C_6$, which is Hamiltonian. Now using induction on $n$ we can easily show that $P_{n}$ is Hamiltonian\cite{Sheu1999,Zaks1984}. It has also been shown that all cycles of lengths ranging between 6 and $n!$ can be embedded on the graph $P_n$\cite{Sheu1999, Kanevsky1995}. The above results play a major role in the method that we propose.\par
\noindent When the diameter of $P_n$ is considered the greatest hindrance in determining a precise value is the complex cyclic structure that these graphs possess. Hence we construct a method that eliminates these cycles and correspondingly calculates the diameter. The diameter calculation for Cayley graphs is an NP-hard problem\cite{Even1981}. Therefore we try to conceive a better bound for the value of the diameter. The absence of cycles is obtained in trees, and the best set of trees connected to a graph is the spanning trees. So we will analyse the spanning trees of $P_n$ and hence conceive a bound for its diameter. But as such, considering every possible spanning tree of the graph is again a headache; hence we need a smaller collection of spanning trees to work on. The following ideas were obtained from this train of thought. The next section contains an overview of existing bounds to the solution. In the following section, we discuss the method proposed and observations made. 
\section*{Literary Review}
The first set of bounds to the solution of the Diameter problem were proposed by William H. Gates and Christos H. Papadimitriou in 1979\cite{Gates1979}. Here they defined two new terminologies to determine their bounds. They were as follows:\par
\noindent Given a permutation $\pi$ in $\S_n$, if $\lvert \pi(j) - \pi(j+1)\rvert = 1$ we call $(j,j+1)$ to be an \textit{adjacency}, where $1 \leq j \leq n$. An adjacency is also seen if ${\pi(j),\pi(j+1)} = {1,n}$. A consecutive string of adjacencies in a permutation is called a \textit{block}.\par
\noindent If for some $1 \leq j \leq n$, neither $(j,j+1)$ nor $(j-1,j)$ are adjacencies then we call $\pi(j)$ to be \textit{free}.\par
\noindent So as to calculate the upperbound for the solution Gates and Papadimitriou created an algorithm. The algorithm, for each permutation, exploits the structural properties given in the above definitions to determine the number of reversals to reach the identity permutation, $I_n$. They calculated the maximum number of prefix reversals required to reach identity given any permutation $\pi$ in $S_n$ and came up with the inequality:
\begin{equation}
    f(n) \leq {5n+5}/3
\end{equation}
where f(n) is the diameter of $P_n$ for a given $n$.
To calculate the lower bound they considered a string of length 8 and showed that the lower bound $f(n)$ to be:
\begin{equation}
    {17n}/16 \leq f(n)
\end{equation}
whenever 16 divides $n$.\par
\noindent But they had also quoted that a better set of bounds could be calculated if the length of string was seven. Working upon this, a better set of bounds was calculated by Hal Sudborough and Heydari in 1997\cite{Heydari1997}:
\begin{equation}
    {15n}/14 \leq f(n)
\end{equation}
whenever 14 divides $n$.\par
\noindent An even better set of bounds were proposed by a team from University of Texas, Dallas and Hal Sudborough in 2009\cite{Chitturi2009}:
\begin{equation}
    f(n) \leq {18n}/11
\end{equation}
The precise values for the diameter of Pancake graphs have been calculated up to $n=19$. 
Heydari and Sudborough also computed the diameter of the Pancake graph $P_n$ up to $n=13$. The diameter for $n=14$,$15$ were calculated in 2005 by Yuusuke Kounoike, Keiichi Kaneka, and Yuji Shinano\cite{Kounoike2005}. In one year, in 2006, the same authors and Shogo Asai determined the diameter for $n=16$,$17$\cite{Asai2006}. In 2011 the diameter for $n=18$,$19$ was found by Cibulka\cite{Cibulka2011}.
The table of the values for the values of $d=diam(P_n)$ for $2 \leq n \leq 19$ is given below:
\begin{table}[H]
\centering
\begin{tabular}{|c|c|c|c|c|c|c|c|c|c|c|c|c|c|c|c|c|c|c|}
\hline
n & 2 & 3 & 4 & 5 & 6 & 7 & 8 & 9 & 10 & 11 & 12 & 13 & 14 & 15 & 16 & 17 & 18 & 19\\
\hline
d & 1 & 3 & 4 & 5 & 7 & 8 & 9 & 10 & 11 & 13 & 14 & 15 & 16 & 17 & 18 & 19 & 20 & 22\\    
\hline
\end{tabular}
\caption{The Table of Diameters}
\label{table 2}
\end{table}
\section*{Observations}
In this section, we will look into the method that we propose. The computational complexity of the method has not been discussed, but the theoretical validity has been shown below the proposed method. As mentioned in the final paragraph of the introduction, we look into the following definition before going into the method.\par
\noindent  A \textit{path covering collection of trees} refers to the collection of trees that contain all paths between a given set of vertices.
Now that we have a collection to work upon, how do we construct such a collection? The following result gives the answer:
\begin{proposition}
 Given a Hamiltonian graph G, A path covering collection of trees is obtained by the action of the automorphism group on a Hamiltonian path in G.
\end{proposition}
\begin{proof}
The proof of the above result is in two parts:
\begin{enumerate}
    \item To show that the action results in a tree of G.
    \item To show that the collection so obtained is path covering.
\end{enumerate}
Let $H$ be the required Hamiltonian path, $L(G)$ be the automorphism group for $G$.
	\begin{enumerate}
		\item Let the image of $H$ under some automorphism $f$ be not a tree. Then there exists some vertex $v_k$ such that the graph has a cycle with respect to $v_k$. This means there exists vertices $h_k$ and $h_s$ in $H$ such that:
            \begin{equation}
                f(h_k)=f(h_s)=v_k
            \end{equation}
	    But $f$ is an automorphism on $G$ hence a bijection. Thus the above statement contradicts the property of $f$ hence our assumption is wrong. Therefore the collection so formed has only trees in it. The fact that they are spanning is also seen by the fact that the elements of $L(G)$ are bijections.
	    \item Consider a path $p$ of length $k$ in $G$, $(v_{1} v_{2}...v_{k+1})$. Consider any $k$ length segment of $H$, 	   	$(h_{1}h_{2}...h_{k+1})$. Choose $f_1$ such that $f_1(v_1)= h_1$. Now choose $f_2$ such that $f_2(f_1(v_1))=h_1$ and 	$f_2(v_2)=h_2$. Inductively continue until the $(k+1)$th vertex. Since the operation between $f_{i}$’s is composition, all the corresponding functions are in $L(G)$. Therefore we see that the collection obtained is path covering.
\end{enumerate}
\end{proof}	
\noindent Hence we have constructed the required collection $\mathbb{S}$. The next question in front of us is the comparison of distances. The distance between two vertices in the spanning tree may not be equal to the distance between two vertices in the graph $G$. The problem is taken care of by the following properties:
\begin{proposition}
 $\min\limits_{S_i \in \mathbb{S}}d_{S_i}(u,v) = d_{G}(u,v)$ where $\mathbb{S}$ is a path covering collection of spanning trees for $G$.
\end{proposition}
\begin{proof}
Let us assume that the above result is not true then either
\begin{enumerate}
\item $\min\limits_{S \in \mathbb{S}}d_{S}(u,v) < d_{G}(u,v)$
\item $\min\limits_{S \in \mathbb{S}}d_{S}(u,v) > d_{G}(u,v)$
\end{enumerate}
	In the first case, we see that the result shows that we have a path shorter than the shortest path in $G$ in some tree $S$, which is impossible as $S$ is a spanning tree of $G$. 
	In the second case, we see that the result shows that some path $p$ in $G$ is not present in any tree $S \in \mathbb{S}$ which is also impossible as the collection $\mathbb{S}$ is path covering as proven above.
	Therefore we infer that $\min\limits_{S_i \in \mathbb{S}}d_{S_i}(u,v) = d_{G}(u,v)$.
\end{proof}
\noindent Hence we have a way to compare distance in the tree and distance in the graph. The first result can be applied to any graph while the second can be applied to any simple graph. The next set of results are for $P_n$.
\begin{proposition}
There exists a permutation $u$ such that $d_{P_n}(u,In)= diam(Pn)$, where $I_n$ corresponds to the identity element of $S_n$.
\end{proposition}
\begin{proof}
By the definition of diameter we see that for any $n$ there exists vertices $u,v \in P_n$ such that $d_{P_n}(u,v)= diam(P_n)$. By the definition of distance on a graph it is easily seen that its a 	metric and hence is translation invariant. Therefore
\begin{equation}
u^{-1}(d(u,v)) = diam(P_n) \implies 
d(I_n,u^{-1}v) = diam(P_n)    
\end{equation}
That is $u^{-1}v$ is the required permutation.
\end{proof}
\noindent Hence we have seen that we do not have to consider all distances on the graph but rather only distances from $I_n$ to every other vertex of the graph.Finally we have the last result which gives us a lower bound on the value of $diam(P_n)$.
\begin{proposition}
$diam(P_{n-1}) \leq diam(P_{n}).$
\end{proposition}
\begin{proof}
Let us assume the above result is not true then as per hierarchy and the result stated 	before  $\exists u \in P_{n-1}(i)$ in $P_n$ such that $u$ is at a 	distance greater than $diam(P_n)$ from $I_n$. This is contradictory to the definition of diameter.
\end{proof}
\noindent Using the above results we fabricate a method to compute a bound for the diameter of $P_n$, $n>5$.
\subsection*{Proposed Method}
The following are the steps to the proposed new method:
\begin{enumerate}
\item Construct a set $M= \{u \in S_n | d_{P_n}(I_n,u) \geq diam(P_{n-1})\}$(The construction of $M$ is explained in the next section).
\item Consider the longest connected segment among the elements of $M$ with respect to the graph $P_n$. Let the segment be $T$.
\item Construct a family of trees with vertices in $M$ using the Dijkstra's algorithm with respect to the graph $P_n$.
\item Calculate using \textit{\textbf{Proposition 2}} the distance between the vertices in $M$ in $P_n$. 
\item Determine the maximum of all distances calculated using the above step. Let that value be $d_n$
\item $diam(P_n) \leq diam(P_{n-1})+d_n$.
\end{enumerate}
\subsection*{Theoretical Validity of the Method}
From the method stated above, we see that the method's efficiency lies in the precision we can attain during the construction of $M$. Hence we consider the construction of $M$ first.
\subsubsection*{Construction of $M$}
The following are the steps to the construction of set $M$:
\begin{enumerate}
\item Divide the value $diam(P_{n-1})$ by 5 and determine the quotient $q$ and remainder $r$.
\item Construct the set $F$ of all five sequences of elements in the generating set of $P_n$ such that no two consecutive elements are the same.
\item Now consider every path of the form of some element in $F$ starting from $I_n$ in $P_n$. Let the set $B_1$ contain all intermediate vertices in the paths from $I_n$ and let $C_1$ be the collection of end vertices in each case.
\item Let set $S^{1} = S_n \setminus B_{1} \cup C_{1}$.
\item Now, with respect to the first vertex in $C_{1}$ repeat step 3. That is, in place of $I_n$, consider the first vertex in $C_{1}$. Construct sets $B_{21}$ and $C_{21}$ accordingly and create set $ S^{21}= S_n\setminus B_{21} \cup C_{21}$.
\item Repeat the step until all vertices in $C_1$ are covered.
\item Define $S^{2}$ to be the intersection of all $S^{2i^{\prime}s}$ and $S^1$.
\item So on defines until Sq.
\item Now consider the vertices in the set $C_{qi^{\prime}s}$.
\item Construct the set of all $r$ length sequences of elements of the generating set. Let it be $F_r$.
\item Now repeat step 3 with the vertices of $C_{qi^{\prime}s}$ and with elements from $F_r$.
\item The vertices that are left after repeating the above steps on $C_{qi^{\prime}s}$ form the set $M$.
\end{enumerate}
\noindent Now that we know the set of $M$, let us consider the next part of the method. The basic question that arises is what happens when the rest of the vertices are totally disconnected. Then we claim that the $diam(P_n)=diam(P_{n-1})$. If not, then we see that $diam(P_n) < diam(P_{n-1})$ which has already been proven wrong. Therefore we are left with the possibility that $diam(P_n)=diam(P_{n-1})$. If the graph left is not totally disconnected, then we consider the longest connected path in the resulting graph to come up with a better bound. From the results in the last section, we see that steps 3 and 4 are valid. All that is left is to consider the final statement. Now we have shown that there exists a path of length $diam(P_{n-1})+d_n$ in the graph.  Hence by the fact that the defined distance on the graph is a metric, we see that:
\begin{equation}
diam(P_n) \leq diam(P_{n-1})+d_n
\end{equation}
\section*{Conclusion}
To conclude, we see that the method provides us with a bound to the diameter of $P_n$. Calculating the diameter depends on the difference between the lower and upper bounds determined. If the number of integer values between these bounds is comparably small, we can calculate the diameter by considering these integral values.
\medskip
\bibliographystyle{plain}
\bibliography{Harigovind}

\begin{thebibliography}{10}

\bibitem{Asai2006}
Shogo Asai, Yuusuke Kounoike, Yuji Shinano, and Keiichi Kaneko.
\newblock Computing the diameter of 17-pancake graph using a pc cluster.
\newblock In {\em European Conference on Parallel Processing}, pages
  1114--1124. Springer, 2006.

\bibitem{Chitturi2009}
Bhadrachalam Chitturi, William Fahle, Zhaobing Meng, Linda Morales, Charles~O
  Shields, Ivan~Hal Sudborough, and Walter Voit.
\newblock An (18/11) n upper bound for sorting by prefix reversals.
\newblock {\em Theoretical Computer Science}, 410(36):3372--3390, 2009.

\bibitem{Cibulka2011}
Josef Cibulka.
\newblock On average and highest number of flips in pancake sorting.
\newblock {\em Theoretical Computer Science}, 412(8-10):822--834, 2011.

\bibitem{Even1981}
Shimon Even and Oded Goldreich.
\newblock The minimum-length generator sequence problem is np-hard.
\newblock {\em Journal of Algorithms}, 2(3):311--313, 1981.

\bibitem{Gates1979}
William~H Gates and Christos~H Papadimitriou.
\newblock Bounds for sorting by prefix reversal.
\newblock {\em Discrete mathematics}, 27(1):47--57, 1979.

\bibitem{Heydari1997}
Mohammad~H Heydari and I~Hal Sudborough.
\newblock On the diameter of the pancake network.
\newblock {\em Journal of Algorithms}, 25(1):67--94, 1997.

\bibitem{Kanevsky1995}
Arkady Kanevsky and Chao Feng.
\newblock On the embedding of cycles in pancake graphs.
\newblock {\em Parallel Computing}, 21(6):923--936, 1995.

\bibitem{Kounoike2005}
Yuusuke Kounoike, Keiichi Kaneko, and Yuji Shinano.
\newblock Computing the diameters of 14-and 15-pancake graphs.
\newblock In {\em 8th International Symposium on Parallel Architectures,
  Algorithms and Networks (ISPAN'05)}, pages 6--pp. IEEE, 2005.

\bibitem{Sheu1999}
Jyh-Jian Sheu, JJM Tan, LH~Hsu, and MY~Lin.
\newblock On the cycle embedding of pancake graphs.
\newblock In {\em Proceedings of 1999 National Computer Symposium}, pages
  C414--C419, 1999.

\bibitem{Zaks1984}
Shmuel Zaks.
\newblock A new algorithm for generation of permutations.
\newblock {\em BIT Numerical Mathematics}, 24(2):196--204, 1984.

\end{thebibliography}
\end{document}